\documentclass[letterpaper,12pt]{amsart}
\usepackage{amsmath,amssymb,amsxtra,amsthm}

\makeatletter
\@namedef{subjclassname@2010}{%
  \textup{2010} Mathematics Subject Classification}
\makeatother

\newtheorem{theorem}{Theorem}[section]
\newtheorem{lemma}[theorem]{Lemma}

\theoremstyle{definition}

\numberwithin{equation}{section}

\begin{document}

\title{On some conjectures concerning Stern's sequence and its twist}
\author{Michael Coons}
\address{University of Waterloo, Dept.~of Pure Math., Waterloo, ON, N2L 3G1, Canada}
\email{mcoons@math.uwaterloo.ca}
\thanks{Research supported by a Fields--Ontario Fellowship.}%

\subjclass[2010]{Primary 11B37; Secondary 11B83}%
\keywords{Stern sequence, functional equations, binary expansion}%
\date{\today}

\begin{abstract}
In a recent paper, Roland Bacher conjectured three identities concerning Stern's sequence and its twist. In this paper we prove Bacher's conjectures. Possibly of independent interest, we also give a way to compute the Stern value (or twisted Stern value) of a number based solely on its binary expansion.
\end{abstract}

\maketitle

\section{Introduction}

We define the {\em Stern sequence} (also known as {\em Stern's diatomic sequence}) $\{s(n)\}_{n\geqslant 0}$ by $s(0)=0$, $s(1)=1$, and for all $n\geqslant 1$ by $$s(2n)=s(n),\qquad s(2n+1)=s(n)+s(n+1);$$ this is sequence A002487 in Sloane's list. We denote by $S(z)$ the generating function of the Stern sequence; that is, $$S(z):=\sum_{n\geqslant 0} s(n)z^n.$$ Stern's sequence has been well studied and has many interesting properties (see e.g., \cite{Dil1, Dil2, Leh1, Lin1} for details). One of the most interesting properties is that the sequence $\{s(n+1)/s(n)\}_{n\geqslant 1}$ is an enumeration of the positive reduced rationals without repeats.

Similarly Bacher \cite{B} introduced the {\em twisted Stern sequence} $\{t(n)\}_{n\geqslant 0}$ given by the recurrences $t(0)=0$, $t(1)=1$, and for $n\geqslant 1$ by $$t(2n)=-t(n),\qquad t(2n+1)=-t(n)-t(n+1).$$ We denote by $T(z)$ the generating function of the twisted Stern sequence; that is, $$T(z):=\sum_{n\geqslant 0} t(n)z^n.$$ Towards describing the relationship between the Stern sequence and its twist, Bacher \cite{B} gave many results, and two conjectures. As the main theorems of this article, we prove these conjectures, so we will state them as theorems (note that we have modified some of the notation). 

\begin{theorem}\label{Bacherconj1} There exists an integral sequence $\{u(n)\}_{n\geqslant 0}$ such that for all $e\geqslant 0$ we have $$\sum_{n\geqslant 0} t(3\cdot 2^e+n)z^n=(-1)^eS(z)\sum_{n\geqslant 0}u(n)z^{n\cdot 2^e}.$$ 
\end{theorem}

Note that in this theorem (as in the original conjecture), it is implicit that the sequence $\{u(n)\}_{n\geqslant 0}$ is defined by the relationship $$U(z):=\sum_{n\geqslant 0}u(n)z^n =\frac{\sum_{n\geqslant 0}t(3+n)z^n}{S(z)}.$$

\begin{theorem}\label{Bacherconj2} (i) The series $$G(z):=\frac{\sum_{n\geqslant 0}(s(2+n)-s(1+n))z^n}{S(z)}$$ satisfies $$\sum_{n\geqslant 0}(s(2^{e+1}+n)-s(2^e+n))z^n=G(z^{2^e})S(z)$$ for all $e\in\mathbb{N}$.

Similarly, (ii) the series $$H(z):=-\frac{\sum_{n\geqslant 0}(t(2+n)+t(1+n))z^n}{S(z)}$$ satisfies $$(-1)^{e+1}\sum_{n\geqslant 0}(t(2^{e+1}+n)+t(2^e+n))z^n=H(z^{2^e})S(z)$$ for all $e\in\mathbb{N}$.
\end{theorem}

These theorems were originally stated as Conjectures 1.3 and 3.2 in \cite{B}.

\section{Untwisting Bacher's First Conjecture}

In this section, we will prove Theorem \ref{Bacherconj1}, but first we note the following lemma which are a direct consequence of the definitions of the Stern sequence and its twist.

\begin{lemma}\label{AT} The generating series $S(z)=\sum_{n\geqslant 0}s(n)z^n$ and $T(z)=\sum_{n\geqslant 0}t(n)z^n$ satisfy the functional equations $$S(z^2)=\left(\frac{z}{1+z+z^2}\right)S(z)$$ and $$T(z^2)=\left(T(z)-2z\right)\left(\frac{-z}{1+z+z^2}\right),$$ respectively.
\end{lemma}

We prove here only the functional equation for $T(z)$. The functional equation for the generating series of the Stern sequence is well--known; for details see, e.g., \cite{CoonsS1, Dil1}.

\begin{proof}[Proof of Lemma \ref{AT}] This is a straightforward calculation using the definition of $t(n)$. Note that \begin{align*} T(z)&= \sum_{n\geq 0}t(2n)z^{2n}+\sum_{n\geq 0}t(2n+1)z^{2n+1}\\
&=-\sum_{n\geq 0}t(n)z^{2n}+t(1)z+\sum_{n\geq 1}t(2n+1)z^{2n+1}\\
&=-T(z^2)+z-\sum_{n\geq 1}t(n)z^{2n+1}-\sum_{n\geq 1}t(n+1)z^{2n+1}\\
&=-T(z^2)+z-zT(z^2)-z^{-1}\sum_{n\geq 1}t(n+1)z^{2(n+1)}\\
&=-T(z^2)+2z-zT(z^2)-z^{-1}\sum_{n\geq 0}t(n+1)z^{2(n+1)}\\
&=-T(z^2)+2z-zT(z^2)-z^{-1}T(z^2).
\end{align*} Solving for $T(z^2)$ gives $$T(z^2)=\left(T(z)-2z\right)\left(\frac{-z}{1+z+z^2}\right),$$ which is the desired result.
\end{proof}

Since the proof of Theorem \ref{Bacherconj1} is easiest for the case $e=1$, and this case is indicative of the proof for the general case, we present it here separately.

\begin{proof}[Proof of Theorem \ref{Bacherconj1} for $e=1$] Recall that we define the sequence $\{u(n)\}_{n\geqslant 0}$ is by the relationship $$U(z):=\sum_{n\geqslant 0}u(n)z^n =\frac{\sum_{n\geqslant 0}t(3+n)z^n}{S(z)}.$$ Since $$\sum_{n\geqslant 0}t(3+n)z^n=\frac{1}{z^3}\left(T(z)+z^2-z\right),$$ we have that \begin{equation}\label{bUA}U(z)=\frac{T(z)+z^2-z}{z^3S(z)}=\frac{1}{z^3}\cdot\frac{T(z)}{S(z)}+\frac{z^2-z}{z^3}\cdot\frac{1}{S(z)}.\end{equation} Note that we are interested in a statement about the function $U(z^{2})$. We will use the functional equations for $S(z)$ and $T(z)$ to examine this quantity via \eqref{bUA}. Note that equation \eqref{bUA} gives, sending $z\mapsto z^{2}$ and using applying Lemma \ref{AT}, that $$U(z^2)=\frac{1}{z^6}\cdot\frac{T(z^2)}{S(z^2)}+\frac{z^4-z^2}{z^6}\cdot\frac{1}{S(z^2)}=\frac{1}{z^6S(z)}\left(2z-T(z)+(z^3-z)(1+z+z^2)\right).$$ Thus we have that \begin{align*} (-1)^1 S(z)U(z^2)&=\frac{-1}{z^6}\left(2z-T(z)-z-z^2-z^3+z^3+z^4+z^5\right)\\
&=\frac{1}{z^6}\left(T(z)-z+z^2-z^4-z^5\right)\\
&=\frac{1}{z^6}\sum_{n\geq 6}t(n)z^n\\
&=\sum_{n\geq 0}t(3\cdot 2+n)z^n,
\end{align*} which is exactly what we wanted to show.
\end{proof}

For the general case, complications arise in a few different places. The first is concerning $T(z^{2^e})$. We will build up the result with a sequence of lemmas to avoid a long and calculation--heavy proof of Theorem \ref{Bacherconj1}.

\begin{lemma}\label{T2e} For all $e\geq 1$ we have $$T(z^{2^e})=T(z)\prod_{i=0}^{e-1}\left(\frac{-z^{2^i}}{1+z^{2^i}+z^{2^{i+1}}}\right)-2\sum_{j=0}^{e-1}z^{2^j}\prod_{i=j}^{e-1}\left(\frac{-z^{2^i}}{1+z^{2^i}+z^{2^{i+1}}}\right).$$
\end{lemma}

\begin{proof} We give a proof by induction. Note that for $e=1$, the right--hand side of the desired equality is $$T(z)\left(\frac{-z}{1+z+z^{2}}\right)-2z\left(\frac{-z}{1+z+z^{2}}\right)=\left(T(z)-2z\right)\left(\frac{-z}{1+z+z^{2}}\right)=T(z^2)$$ where the last equality follows from Lemma \ref{AT}.

Now suppose the identity holds for $e-1$. Then, again using Lemma \ref{AT}, we have 
\begin{align*} T(z^{2^e}) = T((z^2)^{2^{e-1}})&=T(z^2)\prod_{i=0}^{e-2}\left(\frac{-z^{2^{i+1}}}{1+z^{2^{i+1}}+z^{2^{i+2}}}\right)-2\sum_{j=0}^{e-2}z^{2^{j+1}}\prod_{i=j}^{e-2}\left(\frac{-z^{2^{i+1}}}{1+z^{2^{i+1}}+z^{2^{i+2}}}\right)\\
&=\left(T(z)-2z\right)\left(\frac{-z}{1+z+z^{2}}\right)\prod_{i=1}^{e-1}\left(\frac{-z^{2^{i}}}{1+z^{2^{i}}+z^{2^{i+1}}}\right)\\
&\qquad\qquad-2\sum_{j=1}^{e-1}z^{2^{j}}\prod_{i=j}^{e-1}\left(\frac{-z^{2^{i}}}{1+z^{2^{i}}+z^{2^{i+1}}}\right)\\
&=\left(T(z)-2z\right)\prod_{i=0}^{e-1}\left(\frac{-z^{2^{i}}}{1+z^{2^{i}}+z^{2^{i+1}}}\right)-2\sum_{j=1}^{e-1}z^{2^{j}}\prod_{i=j}^{e-1}\left(\frac{-z^{2^{i}}}{1+z^{2^{i}}+z^{2^{i+1}}}\right)\\
&=T(z)\prod_{i=0}^{e-1}\left(\frac{-z^{2^{i}}}{1+z^{2^{i}}+z^{2^{i+1}}}\right)-2\sum_{j=0}^{e-1}z^{2^{j}}\prod_{i=j}^{e-1}\left(\frac{-z^{2^{i}}}{1+z^{2^{i}}+z^{2^{i+1}}}\right).
\end{align*} Hence, by induction, the identity is true for all $e\geq 1$.
\end{proof}

We will need the following result for our next lemma.

\begin{theorem}[Bacher \cite{B}]\label{B1.4} For all $e\geqslant 1$, we have $$\prod_{i=0}^{e-1}\left(1+z^{2^i}+z^{2^{i+1}}\right)=\frac{(-1)^e}{z(1+z^{2^e})}\sum_{n=0}^{3\cdot 2^e}t(3\cdot 2^e+n)z^n.$$
\end{theorem}

The following lemma is similar to the comment made in Remark 1.5 of \cite{B}.

\begin{lemma} For all $e\geq 1$, we have that $$\sum_{n=0}^{3\cdot 2^e}t(n)z^n=z-z^2+\sum_{k=0}^{e-1}(-1)^kz^{3\cdot 2^k+1}(z^{2^k}+1)\prod_{i=0}^{k-1}(1+z^{2^i}+z^{2^{i+1}}).$$
\end{lemma}

\begin{proof} If $e\geq 1$, then we have \begin{align*} \sum_{n=0}^{3\cdot 2^e}t(n)z^n &= z-z^2+\sum_{k=0}^{e-1}\sum_{n=3\cdot 2^k}^{3\cdot 2^{k+1}}t(n)z^n\\
&= z-z^2+\sum_{k=0}^{e-1}\sum_{n=0}^{3\cdot 2^{k}}t(3\cdot 2^{k}+n)z^{n+3\cdot 2^{k}}\\
&= z-z^2+\sum_{k=0}^{e-1}z^{3\cdot 2^k}\sum_{n=0}^{3\cdot 2^{k}}t(3\cdot 2^{k}+n)z^{n}.
\end{align*} Applying Theorem \ref{B1.4}, we have that $$\sum_{n=0}^{3\cdot 2^e}t(n)z^n=z-z^2+\sum_{k=0}^{e-1}z^{3\cdot 2^k}(-1)^kz(z^{2^k}+1)\prod_{i=0}^{k-1}(1+z^{2^i}+z^{2^{i+1}}),$$ which after some trivial term arrangement gives the result.
\end{proof}

\begin{lemma}\label{tspp} For all $e\geq 1$, we have $$\sum_{n=0}^{3\cdot 2^e} t(n)z^n=2z\sum_{j=0}^{e-1}(-1)^j\prod_{i=0}^{j-1}(1+z^{2^i}+z^{2^{i+1}})-(-1)^ez(z^{2^e}-1)\prod_{i=0}^{e-1} (1+z^{2^i}+z^{2^{i+1}}).$$
\end{lemma}

\begin{proof} This lemma is again proved by induction, using the result of the previous lemma. Note that in view of the previous lemma, by subtracting the first term on the right--hand side of the desired equality, it is enough to show that for all $e\geq 1$, we have \begin{multline}\label{lrhs}z-z^2+\sum_{k=0}^{e-1}(-1)^k\left(z^{4\cdot 2^k}+z^{3\cdot 2^k}-2\right)z\prod_{i=0}^{k-1}(1+z^{2^i}+z^{2^{i+1}})\\ =-(-1)^ez(z^{2^e}-1)\prod_{i=0}^{e-1} (1+z^{2^i}+z^{2^{i+1}}).\end{multline}

If $e=1$, then the left--hand side of \eqref{lrhs} is $$z-z^2+(z^4+z^3-s)z=-z-z^2+z^4+z^5,$$ and the right--hand side of \eqref{lrhs} is $$-(-1)z(z^2-1)(1+z+z^2)=-z-z^2+z^4+z^5,$$ so that \eqref{lrhs} holds for $e=1$.

Now suppose that \eqref{lrhs} holds for $e-1$. Then \begin{align*} z-z^2+\sum_{k=0}^{e-1}(-1)^k&\left(z^{4\cdot 2^k}+z^{3\cdot 2^k}-2\right)z\prod_{i=0}^{k-1}(1+z^{2^i}+z^{2^{i+1}})\\
&= (-1)^{e-1}\left(z^{4\cdot 2^{e-1}}+z^{3\cdot 2^{e-1}}-2\right)z\prod_{i=0}^{e-2}(1+z^{2^i}+z^{2^{i+1}})\\
&\qquad\qquad+z-z^2+\sum_{k=0}^{e-2}(-1)^k\left(z^{4\cdot 2^k}+z^{3\cdot 2^k}-2\right)z\prod_{i=0}^{k-1}(1+z^{2^i}+z^{2^{i+1}})\\
&= (-1)^{e-1}\left(z^{4\cdot 2^{e-1}}+z^{3\cdot 2^{e-1}}-2\right)z\prod_{i=0}^{e-2}(1+z^{2^i}+z^{2^{i+1}})\\
&\qquad\qquad-(-1)^{e-1}z(z^{2^{e-1}}-1)\prod_{i=0}^{e-2} (1+z^{2^i}+z^{2^{i+1}})\end{align*} Factoring out the product we thus have that 
\begin{align*}
z-z^2+\sum_{k=0}^{e-1}(-1)^k&\left(z^{4\cdot 2^k}+z^{3\cdot 2^k}-2\right)z\prod_{i=0}^{k-1}(1+z^{2^i}+z^{2^{i+1}})\\
&=(-1)^e\prod_{i=0}^{e-2} (1+z^{2^i}+z^{2^{i+1}})\cdot \left(-\left(z^{4\cdot 2^{e-1}}+z^{3\cdot 2^{e-1}}-2\right)z+z(z^{2^{e-1}}-1)\right)\\
&=-(-1)^ez\prod_{i=0}^{e-2} (1+z^{2^i}+z^{2^{i+1}})\cdot \left(z^{4\cdot 2^{e-1}}+z^{3\cdot 2^{e-1}}-z^{2\cdot 2^{e-1}}-1\right)\\
&=-(-1)^ez\prod_{i=0}^{e-2} (1+z^{2^i}+z^{2^{i+1}})\cdot (z^{2^e}-1)(1+z^{2^{e-1}}+z^{2^{e}})\\
&=-(-1)^ez(z^{2^e}-1)\prod_{i=0}^{e-1} (1+z^{2^i}+z^{2^{i+1}}),
\end{align*} so that by induction, \eqref{lrhs} holds for all $e\geq 1$.
\end{proof}

With these lemmas in place we are in position to prove Theorem \ref{Bacherconj1}.

\begin{proof}[Proof of Theorem \ref{Bacherconj1}] We start by restating \eqref{bUA}; that is $$U(z)=\frac{T(z)+z^2-z}{z^3S(z)}=\frac{1}{z^3}\cdot\frac{T(z)}{S(z)}+\frac{z^2-z}{z^3}\cdot\frac{1}{S(z)}.$$ Sending $z\mapsto z^{2^e},$ we have that $$U(z^{2^e})=\frac{1}{z^{3\cdot 2^e}}\cdot\frac{T(z^{2^e})}{S(z^{2^e})}+\frac{z^{2^{e+1}}-z^{2^e}}{z^{3\cdot 2^e}}\cdot\frac{1}{S(z^{2^e})}=\frac{1}{z^{3\cdot 2^e}}\cdot\frac{T(z^{2^e})}{S(z^{2^e})}+\frac{z^{2^{e+1}}-z^{2^e}}{z^{3\cdot 2^e}z^{2^e-1}S(z)}\cdot\prod_{i=0}^{e-1} (1+z^{2^i}+z^{2^{i+1}}),$$ where we have used the functional equation for $S(z)$ to give the last equality. Using Lemma \ref{T2e} and the functional equation for $S(z)$, we have that \begin{align}\nonumber\frac{T(z^{2^e})}{S(z^{2^e})}&=\frac{T(z)\prod_{i=0}^{e-1}\left(\frac{-z^{2^i}}{1+z^{2^i}+z^{2^{i+1}}}\right)-2\sum_{j=0}^{e-1}z^{2^j}\prod_{i=j}^{e-1}\left(\frac{-z^{2^i}}{1+z^{2^i}+z^{2^{i+1}}}\right)}{S(z)}\cdot\prod_{i=0}^{e-1} \left(\frac{1+z^{2^i}+z^{2^{i+1}}}{z^{2^i}}\right)\\
\label{ToverS}&=(-1)^e\frac{T(z)}{S(z)}-(-1)^e\frac{2z}{S(z)}\sum_{j=0}^{e-1}(-1)^{j}\prod_{i=0}^{j-1}\left({1+z^{2^i}+z^{2^{i+1}}}\right).\end{align} Applying this to the expression for $U(z^{2^e})$ we have, multiplying by $(-1)^eS(z)$, that \begin{multline*}(-1)^eS(z)U(z^{2^e})=\frac{1}{z^{3\cdot 2^e}}\left(T(z)-2z\sum_{j=0}^{e-1}(-1)^{j}\prod_{i=0}^{j-1}\left({1+z^{2^i}+z^{2^{i+1}}}\right)\right.\\ \left.+(-1)^ez(z^{2^e}-1)\prod_{i=0}^{e-1}\left({1+z^{2^i}+z^{2^{i+1}}}\right)\right).\end{multline*} Now by Lemma \ref{tspp}, this reduces to $$(-1)^eS(z)U(z^{2^e})=\frac{1}{z^{3\cdot 2^e}}\left(T(z)-\sum_{n=0}^{3\cdot 2^e} t(n)z^n\right)=\sum_{n\geq 0}t(3\cdot 2^3+n)z^n,$$ which proves the theorem.
\end{proof}




\section{Untwisting Bacher's Second Conjecture}

In this section, we will prove Theorem \ref{Bacherconj2}. For ease of reading we have separated the proofs of the two parts of Theorem \ref{Bacherconj2}.

To prove Theorem \ref{Bacherconj2}(i) we will need the following lemma.

\begin{lemma}\label{pss} For all $k\geq 0$ we have that $$z\prod_{i=0}^{k-1}\left(1+z^{2^i}+z^{2^{i+1}}\right)=\sum_{n=1}^{2^k}s(n)z^n+\sum_{n=1}^{2^k-1}s(2^k-n)z^{n+2^k}.$$
\end{lemma}

\begin{proof} Again, we prove by induction. Note that for $k=0$, the product and the left--most sum are both empty, thus they are equal to $1$ and $0$, respectively. Since $$z=s(1)z=\sum_{n=1}^{2^0}s(n)z^n$$ the theorem is true for $k=0$. To use some nonempty terms, we consider the case $k=1$. Then we have $$z\prod_{i=0}^{1-1}\left(1+z^{2^i}+z^{2^{i+1}}\right)=z+z^2+z^3=\sum_{n=1}^{2^1}s(n)z^n+\sum_{n=1}^{2^1-1}s(2^1-n)z^{n+2^1},$$ so the theorem holds for $k=1$.

Now suppose the theorem holds for $k-1$. Then \begin{align*} z\prod_{i=0}^{k-1}\left(1+z^{2^i}+z^{2^{i+1}}\right)&=\left(1+z^{2^{k-1}}+z^{2^{k}}\right)\cdot z\prod_{i=0}^{k-2}\left(1+z^{2^i}+z^{2^{i+1}}\right)\\
&=\left(1+z^{2^{k-1}}+z^{2^{k}}\right)\left(\sum_{n=1}^{2^{k-1}}s(n)z^n+\sum_{n=1}^{2^{k-1}-1}s(2^{k-1}-n)z^{n+2^{k-1}}\right)\\
&=\left(\sum_{n=1}^{2^{k-1}}s(n)z^n+\sum_{n=1}^{2^{k-1}-1}s(2^{k-1}-n)z^{n+2^{k-1}}+\sum_{n=1}^{2^{k-1}}s(n)z^{n+2^{k-1}}\right)\\
&\qquad +\left(\sum_{n=1}^{2^{k-1}-1}s(2^{k-1}-n)z^{n+2^{k}}+\sum_{n=1}^{2^{k-1}}s(n)z^{n+2^{k}}\right.\\
&\qquad\qquad\left.+\sum_{n=1}^{2^{k-1}-1}s(2^{k-1}-n)z^{n+3\cdot 2^{k-1}}\right)\\
&=\Sigma_1+\Sigma_2,
\end{align*} where $\Sigma_1$ and $\Sigma_2$ represent the triplets of sums from the previous line (we have grouped the last sums in triplets since we will deal with them that way. Note that we have \begin{multline*}\Sigma_1=\sum_{n=1}^{2^{k-1}}s(n)z^n+s(2^{k-1})z^{2^k}+\sum_{n=1}^{2^{k-1}-1}\left(s(n)+s(2^{k-1}-n)\right)z^{n+2^{k-1}}\\ =\sum_{n=1}^{2^{k-1}}s(n)z^n+s(2^{k})z^{2^k}+\sum_{n=1}^{2^{k-1}-1}s(2^{k-1}+n)z^{n+2^{k-1}}=\sum_{n=1}^{2^{k}}s(n)z^n,\end{multline*} where we have used the fact that $s(2n)=s(n)$ and for $n\in[0,2^j]$ the identity $s(2^j+n)=s(2^j-n)+s(n)$ holds (see, e.g., \cite[Theorem 1.2(i)]{B} for details). Similarly, since $2^{k-1}-n=2^k-(n+2^{k-1})$ and $$s(2^{k-1}-n)+s(n)=s(2^{k-1}+n)=s(2^k+n)-s(n)=s(2^k-n)$$  (see Proposition 3.1(i) and Theorem 1.2(i) of \cite{B}), we have that \begin{align*}\Sigma_2&=\sum_{n=1}^{2^{k-1}-1}\left(s(2^{k-1}-n)+s(n)\right)z^{n+2^k}+s(2^{k-1})z^{3\cdot 2^{k-1}}+\sum_{n=1}^{2^{k-1}-1}s(2^{k-1}-n)z^{n+2^{k-1}+2^k}\\ &= \sum_{n=1}^{2^{k-1}-1}s(2^{k}-n)z^{n+2^k}+s(2^{k-1})z^{3\cdot 2^{k-1}}+\sum_{n=2^{k-1}-1}^{2^{k}-1}s(2^{k}-n)z^{n+2^k}\\ &=\sum_{n=1}^{2^k-1}s(2^k-n)z^{n+2^k}.\end{align*} Thus $$\Sigma_1+\Sigma_2=\sum_{n=1}^{2^{k}}s(n)z^n+\sum_{n=1}^{2^k-1}s(2^k-n)z^{n+2^k},$$ and by induction the lemma is proved.
\end{proof}

\begin{proof}[Proof of Theorem \ref{Bacherconj2}(i)] We denote as before the generating series of the Stern sequence by $S(z)$. Splitting up the sum in the definition of $G(z)$ we see that $$G(z)=\frac{1}{z^{2}}(1-z)-\frac{1}{z S(z)},$$ so that using the functional equation for $S(z)$ we have \begin{align*}G(z^{2^e}) = \frac{1}{z^{2^{e+1}}}(1-z^{2^e})-\frac{1}{z^{2^e} S(z^{2^e})}= \frac{1}{z^{2^{e+1}}}(1-z^{2^e})-\frac{1}{z^{2^{e}}}\cdot\frac{\prod_{i=0}^{e-1}(1+z^{2^i}+z^{2^{i+1}})}{z^{2^{e}-1}S(z)}.
\end{align*} This gives \begin{equation}\label{32rhsi} G(z^{2^e})S(z)=\frac{1}{z^{2^{e+1}}}\left((1-z^{2^e})S(z)-z\prod_{i=0}^{e-1}(1+z^{2^i}+z^{2^{i+1}})\right).\end{equation}  We use the previous lemma to deal with the right--hand side of \eqref{32rhsi}; that is, the previous lemma gives that 
\begin{align*}
(1-z^{2^e})S(z)-z\prod_{i=0}^{e-1}(1+z^{2^i}+z^{2^{i+1}})&= \sum_{n\geq 1}s(n)z^n-\sum_{n=1}^{2^e}s(n)z^n\\
&\qquad\qquad-\sum_{n\geq 1}s(n)z^{n+2^e}-\sum_{n=1}^{2^e-1}s(2^e-n)z^{n+2^e}\\
&=\sum_{n\geq 2^e+1} s(n)z^n-\sum_{n\geq 2^e}s(n)z^{n+2^e}\\
&\qquad\qquad-\sum_{n=1}^{2^e-1}(s(n)+s(2^e-n))z^{n+2^e}\\
&=\sum_{n\geq 1} s(2^e+n)z^{n+2^e}-\sum_{n\geq 2^e}s(n)z^{n+2^e}\\
&\qquad\qquad-\sum_{n=1}^{2^e-1}s(2^e+n)z^{n+2^e}\\
&=\sum_{n\geq 0} s(2^{e+1}+n)z^{n+2^{e+1}}-\sum_{n\geq 0}s(2^e+n)z^{n+2^{e+1}}.
\end{align*} Dividing the last line by $z^{2^{e+1}}$ gives the desired result. This proves the theorem.
\end{proof}


The proof of the second part of the theorem follows similarly. We will use the following lemma.

\begin{lemma}[Bacher \cite{B}] For $n$ satisfying $1\leq n\leq 2^e$ we have that 
\begin{enumerate}
\item[(i)] $t(2^{e+1}+n)+t(2^{e}+n)=(-1)^{e+1}s(n),$
\item[(ii)] $t(2^{e}+n)=(-1)^e(s(2^e-n)-s(n))$, 
\item[(iii)] $t(2^{e+1}+n)=(-1)^{e+1}s(2^{e}-n)$.
\end{enumerate}
\end{lemma}

\begin{proof} Parts (i) and (ii) are given in Proposition 3.1 and Theorem 1.2 of \cite{B}, respectively. Part (iii) follows easily from (i) and (ii). 

Note that (i) gives that $$t(2^{e+1}+n)=(-1)^{e+1}s(n)-t(2^{e}+n),$$ which by (ii) becomes \begin{equation*}t(2^{e+1}+n)=(-1)^{e+1}s(n)+(-1)^{e+1}s(2^e-n)-(-1)^{e+1}s(n)=(-1)^{e+1}s(2^e-n).\qedhere\end{equation*} 
\end{proof}

\begin{proof}[Proof of Theorem \ref{Bacherconj2} (ii)] We denote as before the generating series of the Stern sequence by $S(z)$. Splitting up the sum in the definition of $H(z)$ we see that $$H(z)=\frac{1}{z S(z)}-\frac{1+z}{z^{2}}\cdot\frac{T(z)}{S(z)}.$$ Since we will need to consider $H(z^{2^e})$, we will need to compute $\frac{T(z^{2^e})}{S(z^{2^e})}.$ Fortunately we have done this in the proof of Theorem \ref{Bacherconj1}, in \eqref{ToverS}, and so we use this expression here. Thus, applying the functional equation for $S(z)$, we have that \begin{multline*} H(z^{2^e})=\frac{\prod_{i=0}^{e-1}\left(1+z^{2^i}+z^{2^{i+1}}\right)}{z^{2^{e+1}-1}S(z)}-(-1)^e\left(\frac{1+z^{2^e}}{z^{2^{e+1}}}\right)\frac{T(z)}{S(z)}\\ +(-1)^e\left(\frac{1+z^{2^e}}{z^{2^{e+1}}}\right)\frac{2z}{S(z)}\sum_{j=0}^{e-1}(-1)^j\prod_{i=0}^{j-1}\left(1+z^{2^i}+z^{2^{i+1}}\right),\end{multline*} so that \begin{multline*} (-1)^{e+1}H(z^{2^e})S(z)=\frac{1}{z^{2^{e+1}}}\left((1+z^{2^e})T(z)-(-1)^ez\prod_{i=0}^{e-1}\left(1+z^{2^i}+z^{2^{i+1}}\right)\right.\\ \left.-2z(1+z^{2^e})\sum_{j=0}^{e-1}(-1)^j\prod_{i=0}^{j-1}\left(1+z^{2^i}+z^{2^{i+1}}\right)\right).\end{multline*} An application of Lemma \ref{tspp} gives \begin{align*} (-1)^{e+1}H(z^{2^e})S(z)&=\frac{1}{z^{2^{e+1}}}\left((1+z^{2^e})T(z)-(-1)^ez\prod_{i=0}^{e-1}\left(1+z^{2^i}+z^{2^{i+1}}\right)\right.\\ 
&\ \ \left.-(1+z^{2^e})\sum_{n=0}^{3\cdot 2^e}t(n)z^n-(-1)^ez(z^{2^e}-1)(1+z^{2^e})\prod_{i=0}^{e-1}\left(1+z^{2^i}+z^{2^{i+1}}\right)\right)\\
&=\frac{1}{z^{2^{e+1}}}\left(\mathfrak{S}_1+\mathfrak{S}_2+\mathfrak{S}_3+\mathfrak{S}_3\right),\end{align*} where we have used the $\mathfrak{S}_i$ to indicate the terms in the previous line. 

Note that \begin{align*}\mathfrak{S}_1+\mathfrak{S}_3&=(1+z^{2^e})\left(T(z)-\sum_{n=0}^{3\cdot 2^e}t(n)z^n\right)\\
&=(1+z^{2^e})\sum_{n\geq 1}t(3\cdot 2^e+n)z^{n+3\cdot2^e}\\
&=\sum_{n\geq 2^e+1}t(2^{e+1}+n)z^{n+2^{e+1}}+\sum_{n\geq 2^{e+1}+1}t(2^{e}+n)z^{n+2^{e+1}},\end{align*} so that $$\frac{\mathfrak{S}_1+\mathfrak{S}_3}{z^{2^e+1}}=\sum_{n\geq 2^e+1}t(2^{e+1}+n)z^{n}+\sum_{n\geq 2^{e+1}+1}t(2^{e}+n)z^{n}.$$ Using Lemma \ref{pss}, we have \begin{align*} \frac{\mathfrak{S}_2+\mathfrak{S}_4}{z^{2^{e+1}}}&=\frac{1}{z^{2^{e+1}}}\left(-(-1)^e-(-1)^e(z^{2^{e+1}}-1)\right)z\prod_{i=0}^{e-1}\left(1+z^{2^i}+z^{2^{i+1}}\right)\\
&=\frac{(-1)^{e+1}}{z^{2^{e+1}}}\cdot z^{2^{e+1}}\left(\sum_{n=1}^{2^e}s(n)z^n+\sum_{n=1}^{2^e-1}s(2^e-n)z^{n+2^e}\right)\\
&=(-1)^{e+1}\left(\sum_{n=1}^{2^e}s(n)z^n+\sum_{n=1}^{2^e-1}s(2^e-n)z^{n+2^e}\right).
\end{align*} Using the proceeding lemma and the fact that $t(2^{e+1})+t(2^e)=0$ so that we can add in a zero term, we have that \begin{align*} \frac{\mathfrak{S}_2+\mathfrak{S}_4}{z^{2^{e+1}}}&=\sum_{n=1}^{2^e}(t(2^{e+1}+n)+t(2^e+n))z^n+\sum_{n=1}^{2^e-1}t(2^{e+1}+n)z^{n+2^e}\\
&=\sum_{n=1}^{2^e}(t(2^{e+1}+n)+t(2^e+n))z^n+\sum_{n=2^e}^{2^{e+1}-1}t(2^{e}+n)z^{n}\\
&=\sum_{n=0}^{2^e}(t(2^{e+1}+n)+t(2^e+n))z^n+\sum_{n=2^e}^{2^{e+1}-1}t(2^{e}+n)z^{n}\\
&=\sum_{n=0}^{2^e}t(2^{e+1}+n)z^n+\sum_{n=0}^{2^{e+1}-1}t(2^{e}+n)z^{n}.
\end{align*} Putting together these results gives \begin{align*}(-1)^{e+1}H(z^{2^e})S(z)&=\frac{1}{z^{2^{e+1}}}\left(\mathfrak{S}_1+\mathfrak{S}_2+\mathfrak{S}_3+\mathfrak{S}_3\right)\\ &=\sum_{n\geq 0}t(2^{e+1}+n)z^n+\sum_{n\geq 0}t(2^{e}+n)z^{n},\end{align*} which proves the theorem.
\end{proof}

\section{Computing with binary expansions}

To gain intuition regarding Bacher's conjectures mentioned in the first section, we found it very useful to understand what happens to the Stern sequence and its twist at sums of powers of $2$. Thus, in this section we prove the following theorem which removes the need to use the recurrences to give the values of the Stern sequence. 

\begin{theorem} Let $n\geqslant 4$ and write $n=\sum_{i=0}^m 2^ib_i$, the binary expansion of $n$. Then $$s(n)=\left[\begin{matrix}1 & 1\end{matrix}\right]\left(\prod_{i=1}^{m-1}\left[\begin{matrix}1 & 1-b_i\\ b_i & 1 \end{matrix}\right]\right)\left[\begin{matrix}1\\ b_0 \end{matrix}\right]$$
\end{theorem}

\begin{proof} Note that from the definition of the Stern sequence we easily have that $$s(2a+x)=s(a)+x\cdot s(a+1)\qquad (x\in\{0,1\}).$$ It follows that for each $k\in\{0,1,\ldots,m-1\}$ we have that both $$s\left(\sum_{i=k}^m 2^{i-k}b_i\right)=s\left(\sum_{i=k+1}^m 2^{i-(k+1)}b_i\right)+b_k\cdot s\left(1+\sum_{i=k+1}^m 2^{i-(k+1)}b_i\right)$$ and $$s\left(1+\sum_{i=k}^m 2^{i-k}b_i\right)=(1-b_k)\cdot s\left(\sum_{i=k+1}^m 2^{i-(k+1)}b_i\right)+s\left(1+\sum_{i=k+1}^m 2^{i-(k+1)}b_i\right).$$ Starting with $k=0$ and applying the above equalities, and using the fact that $b_m=1$ so that $s(b_m)=s(b_m+1)=1$ gives the result.
\end{proof}

We have a similar result for the twisted Stern sequence whose proof is only trivially different from the above and so we have omitted it.

\begin{theorem} Let $n\geqslant 4$ and write $n=\sum_{i=0}^m 2^ib_i$, the binary expansion of $n$. Then $$t(n)=(-1)^m\left[\begin{matrix}1 & -1\end{matrix}\right]\left(\prod_{i=1}^{m-1}\left[\begin{matrix}1 & 1-b_i\\ b_i & 1 \end{matrix}\right]\right)\left[\begin{matrix}1\\ b_0 \end{matrix}\right].$$
\end{theorem}

Indeed, both $s(n)$ and $t(n)$ are $2$--regular, and so the fact that $s(n)$ and $t(n)$ satisfy theorems like the two above is provided by Lemma 4.1 of \cite{AS} (note that while the existence is proven, the matrices are not explicitly given there).


\section{Acknowledgement} 

The author would like to thank Cameron L.~Stewart for a very enlightening conversation. 



\begin{thebibliography}{6}

\bibitem{AS}
J.-P.~Allouche and J.~Shallit, {The ring of $k$--regular sequences}, {\em Theoret. Comput. Sci.} \textbf{98} (1992), no.~2, 163--197.

\bibitem{B}
R.~Bacher, {Twisting the {S}tern sequence}, preprint, {\tt http://arxiv.org/pdf/1005.5627}.

\bibitem{CoonsS1}
M.~Coons, {The transcendence of series related to stern's diatomic
  sequence}, {\em Int. J. Number Theory} \textbf{6} (2010), no.~1, 211--217.

\bibitem{Dil1}
K.~Dilcher and K.~B. Stolarsky, {A polynomial analogue to the {S}tern
  sequence}, {\em Int. J. Number Theory} \textbf{3} (2007), no.~1, 85--103.

\bibitem{Dil2}
K.~Dilcher and K.~B. Stolarsky, {{S}tern polynomials and double--limit continued fractions}, {\em Acta Arith.} \textbf{140} (2009), 119--134.

\bibitem{Leh1}
D.~N. Lehmer, {On {S}tern's diatomic series.}, {\em Amer. Math. Monthly}
  \textbf{36} (1929), 59--67.

\bibitem{Lin1}
D.~A. Lind, {An extension of {S}tern's diatomic series}, {\em Duke Math. J.}
  \textbf{36} (1969), 55--60.

\end{thebibliography}
\end{document}